\documentclass[10pt]{amsart}
\usepackage{enumerate,amsmath,amssymb,latexsym,
amsfonts, amsthm, amscd, mathabx}

\theoremstyle{plain}

\newtheorem{theorem}{Theorem}

\numberwithin{equation}{section}

\newcommand{\PP}{\mathbf{P}}
\newcommand{\PPi}{\mathbf{\Pi}}
\newcommand{\LLL}{\mathbf{L}}
\newcommand{\h}{\hash}

\addtocounter{section}{-1}

\begin{document}

\title {Projective Space: Points and Planes}

\date{}

\author[P.L. Robinson]{P.L. Robinson}

\address{Department of Mathematics \\ University of Florida \\ Gainesville FL 32611  USA }

\email[]{paulr@ufl.edu}

 \keywords{}

\begin{abstract}

We take points and planes as fundamental, lines as derived, in an axiomatic formulation of three-dimensional projective space, the self-dual nature of which formulation renders automatic the principle of duality. 

\end{abstract}

\maketitle

\medbreak

\section*{Introduction} 

Traditional axiomatic approaches to three-dimensional projective space start from points and lines as fundamental elements and from them construct planes as derived elements. Accordingly, the {\it primary} notion `point' and the {\it secondary} notion `plane' are paired in the principle of duality, the proof of which thereby involves some slight awkwardness. In [R] we detailed an equivalent axiomatization that takes line alone as fundamental element, from which point and plane are derived in parallel: the principle of duality (which pairs the secondary notions) requires no proof, as the axioms are self-dual. Our purpose here is to pursue an alternative course that naturally presents itself: we take point and plane as fundamental, line as derived; again the axioms are self-dual, so the principle of duality (which pairs the primary notions) is automatic. 

\medbreak 

\section*{The Axiomatic Formulation}

\medbreak 

We begin with a pair of disjoint nonempty sets: $\PP$, whose elements we call points and denote by upper case Roman letters; $\PPi$, whose elements we call planes and denote by lower case Greek letters. Between these sets we posit a relation $\h$ of incidence: when $P \in \PP$ and $\pi \in \PPi$ are incident, we write $P \h \pi$ and $\pi \h P$ indiscriminately, saying that $P$ lies in $\pi$ or that $\pi$ passes through $P$; we may also say that $P$ is on $\pi$ or that $\pi$ is on $P$. Some additional notation will simplify the presentation of our axioms. 

\medbreak 

When $P \in \PP$ we define 
$$P^{\h} = \{ \pi \in \PPi : P \h \pi \}$$
and when $\pi \in \PPi$ we define 
$$\pi^{\h} = \{ P \in \PP : \pi \h P \}.$$
More generally, when $\mathbf{S} \subseteq \PP$ we shall write 
$$\mathbf{S}^{\h} = \{ \pi \in \PPi : (\forall P \in \mathbf{S}) P \h \pi \}$$
and when $\mathbf{\Sigma} \subseteq \PPi$ we shall write 
$$\mathbf{\Sigma}^{\h} =  \{ P \in \PP : (\forall \pi \in \mathbf{\Sigma}) \pi \h P \}.$$
Further, it will be convenient to write $\mathbf{S} \h \mathbf{\Sigma}$ to mean that $P \h \pi$ whenever $P \in \mathbf{S}$ and $\pi \in \mathbf{\Sigma}$; for example, $\mathbf{S} \h \mathbf{S}^{\h}$ and $\mathbf{\Sigma} \h \mathbf{\Sigma}^{\h}$. By definition, 
$$\mathbf{S} \subseteq \mathbf{\Sigma}^{\h} \Leftrightarrow \mathbf{S} \h \mathbf{\Sigma} \Leftrightarrow \mathbf{S}^{\h} \supseteq \mathbf{\Sigma}.$$

\medbreak 

The incidence relation $\h$ satisfies the usual `Galois' properties, among which are the folowing: 
$$\mathbf{S_1} \subseteq \mathbf{S_2} \subseteq \PP \Rightarrow \mathbf{S_1}^{\h} \supseteq \mathbf{S_2}^{\h}$$
$$\mathbf{\Sigma_1} \subseteq \mathbf{\Sigma_2} \subseteq \PPi \Rightarrow \mathbf{\Sigma_1}^{\h} \supseteq \mathbf{\Sigma_2}^{\h}$$ 
$$\mathbf{S} \subseteq \PP \Rightarrow \mathbf{S} \subseteq \mathbf{S}^{\h \h}$$ 
$$\mathbf{\Sigma} \subseteq \PPi \Rightarrow \mathbf{\Sigma} \subseteq \mathbf{\Sigma}^{\h \h}.$$

\medbreak 

We are now able to present our axioms for three-dimensional projective space, as follows. 

\medbreak 

\noindent 
AXIOM [1] \\ 
If $P \in \PP$ then $P^{\h} \ne \PPi$. \\ 
If $\pi \in \PPi$ then $\pi^{\h} \ne \PP$. 

\medbreak 

\noindent 
AXIOM [2] \\ 
If $A, B \in \PP$ then $| \{ A, B \}^{\h} | > 2$. \\ 
If $\alpha, \beta \in \pi$ then  $| \{ \alpha, \beta \}^{\h} | > 2.$

\medbreak 

\noindent 
AXIOM [3] \\
If $A, B, C \in \PP$ then $\{ A, B, C \}^{\h} \ne \emptyset.$ \\
If $\alpha, \beta, \gamma \in \PPi$ then $\{ \alpha, \beta, \gamma \}^{\h} \ne \emptyset.$

\medbreak 

\noindent 
AXIOM [4] \\
Let $A \ne B$ in $\PP$ and $\alpha \ne \beta$ in $\PPi$. \\
 If $\{ A, B \} \h \{\alpha, \beta \}$ then $\{ A, B \}^{\h} = \{ \alpha, \beta \}^{\h \h}$ and $\{ \alpha, \beta \}^{\h} = \{ A, B \}^{\h \h}.$

\medbreak 

Notice that we may instead begin with the assumption that the disjoint union $\PP \cup \PPi$ is nonempty: if either $\PP$ or $\PPi$ is nonempty then so is the other by AXIOM [2] or AXIOM [3]; indeed, each of $\PP$ and $\PPi$ has at least three elements. Notice also that in AXIOM [4] only the (equivalent) inclusions $\{ A, B \}^{\h} \subseteq \{ \alpha, \beta \}^{\h \h}$ and $\{ \alpha, \beta \}^{\h} \subseteq \{ A, B \}^{\h \h}$ are new; the reverse inclusions hold already by the mere definition of incidence. 

\medbreak 

Now, let $A \ne B$ in $\PP$ and let $C \in \PP$ also. We claim that $\{ A, B \}^{\h} \subseteq C^{\h} \Leftrightarrow C \in \{ A, B \}^{\h \h}$: on the one hand, if $\{ A, B \}^{\h} \subseteq C^{\h}$ then $C \in C^{\h \h} \subseteq \{ A, B \}^{\h \h}$ on account of the Galois properties of incidence; on the other hand, if $C \in \{ A, B \}^{\h \h}$ and $\pi \in \{ A, B \}^{\h}$ then $C \h \pi$ so that $\pi \in C^{\h}$ by definition. We may say that $C$ is {\it collinear} with $A$ and $B$ precisely when either (hence each) of these equivalent conditions is satisfied. The set $\{ A, B \}^{\h \h}$ comprising all points collinear with $A$ and $B$ is then traditionally a {\it pencil of points}. A parallel analysis applies to $\alpha \ne \beta$ in $\PPi$: the set $\{ \alpha, \beta \}^{\h \h}$ comprising all planes {\it collinear} with $\alpha$ and $\beta$ is traditionally a {\it pencil of planes}. 

\medbreak 

In order to define a line in our geometry, we combine these pencils. Thus: a {\it line} is a subset of $\PP \cup \PPi$ having the form $\{ A, B \}^{\h \h} \cup \{ \alpha, \beta \}^{\h \h}$ for $A \ne B$ in $\PP$ and $\alpha \ne \beta$ in $\PPi$ such that $\{ A, B \} \h  \{ \alpha, \beta \}$; notice that AXIOM [4] allows us to write this line equally as $\{ \alpha, \beta \}^{\h} \cup \{ A, B \}^{\h}$. For simplicity, we shall denote this line by either $AB$ or $\alpha \beta$ according to convenience. This line may also be described as $CD$ for any $C \ne D$ in $\{ \alpha, \beta \}^{\h}$: indeed, AXIOM [4] gives $\{C, D \}^{\h \h} = \{ \alpha, \beta \}^{\h} = \{ A, B \}^{\h \h}$; likewise, the line may also be described as $\gamma \delta$ for any $\gamma \ne \delta$ in $\{ A, B \}^{\h}$. Elements of $\{ A, B \}^{\h \h} = \{ \alpha, \beta \}^{\h}$ may be called points of the line $AB$. 

\medbreak 

AXIOM [3] says nothing new when $C \in  \{ A, B \}^{\h \h}$: in this case, $\{ A, B, C \}^{\h} = \{A, B \}^{\h}$ is already nonempty by AXIOM [2]. Accordingly, AXIOM [3] may essentially be replaced by the following theorem together with its counterpart for planes. 

\medbreak 

\begin{theorem} \label{!plane}
If the points $A, B, C$ are not collinear then $| \{ A, B, C \}^{\h} | = 1$. 
\end{theorem} 

\begin{proof} 
AXIOM [3] insists that $\{ A, B, C \}^{\h}$ is nonempty. If $\sigma \ne \tau \in \{ A, B, C \}^{\h}$ then $A, B, C$ are points of the line $\sigma \tau$. 
\end{proof} 

\medbreak 

Otherwise said, three given non-collinear points lie in a unique plane; likewise, three given non-collinear planes pass through a unique point. 

\medbreak 

Notice that our axioms are invariant under the interchange $\PP \leftrightarrow \PPi$. It follows that if points and planes are interchanged in any theorem, the result is a theorem (whose proof is obtained simply by interchange of points and planes in the proof of the original). This is the principle of {\it duality}, according to which the elements `point' and `plane' are {\it dual} while `line' is {\it self-dual}. For example, we have just seen this principle in action following Theorem \ref{!plane}. 

\medbreak 

We shall write $\LLL$ for the set of lines and shall also follow the custom of labelling individual lines by lower case Roman letters: when $\ell \in \LLL$ we may write $\ell_{\PP} = \ell \cap \PP$ and $\ell_{\PPi} = \ell \cap \PPi$; thus $\ell = \ell_{\PP} \cup \ell_{\PPi}$ where $\ell_{\PPi} = \ell_{\PP}^{\h}$ and $\ell_{\PP} = \ell_{\PPi}^{\h}$. The notion of incidence is extended to include lines according to the following (appropriately dual) definitions. Line $\ell$ is incident to (or passes through) point $P$ precisely when $P \in \ell$; equivalently, $P \h \ell_{\PPi}$. Line $\ell$ is incident to (or lies in) plane $\pi$ precisely when $\pi \in \ell$; equivalently, $\pi \h \ell_{\PP}$. Line $\ell'$ is incident to (or meets) line $\ell''$ precisely when $\ell' \cap \ell'' \ne \emptyset$; we claim that $\ell'$ and $\ell''$ then share both a point and a plane. 

\medbreak 

\begin{theorem} \label{meet}
If $\ell'$ and $\ell''$ are lines, then $\ell'_{\PP} \cap \ell''_{\PP} \ne \emptyset \Leftrightarrow \ell'_{\PPi} \cap \ell''_{\PPi} \ne \emptyset.$
\end{theorem} 

\begin{proof} 
If $\ell' = \ell''$ then there is nothing to prove; so suppose $\ell' \ne \ell''$. Let $\ell' = A' B'$ and $\ell'' = A'' B''$. Assume $P \in A' B' \cap A'' B''$; without loss, say $P \ne A'$ and $P \ne A''$. It follows that $P A' = A' B'$ and $P A'' = A'' B''$; it also follows that $P, A'$ and $A''$ are not collinear, for $A' B' \ne A'' B''$. Theorem \ref{!plane} passes a unique plane $\pi$ through $P, A'$ and $A''$; thus $\pi \in P A' \cap P A'' = A' B' \cap A'' B''$. The converse direction is the dual. 
\end{proof} 

\medbreak

A closer inspection of the proof shows that if the incident lines $\ell'$ and $\ell''$ are distinct then $| \ell'_{\PP} \cap \ell''_{\PP} | = 1$ and $| \ell'_{\PPi} \cap \ell''_{\PPi} | = 1$: thus, $\ell'$ and $\ell''$ share a unique point and a unique plane. 

\medbreak 

The next result is neither a surprise nor crucial to our development, but is included for its intrinsic interest. 

\medbreak  

\begin{theorem} \label{proper}
If $B \ne C$ in $\alpha^{\h}$ then $\{ B,  C \}^{\h \h}$ is a proper subset of $\alpha^{\h}$. 
\end{theorem} 

\begin{proof} 
$\{ B,  C \}^{\h \h}$ is a subset of $\alpha^{\h}$ because of the Galois properties of incidence. AXIOM [1] furnishes a point $P \notin \alpha^{\h}$ while AXIOM [2] and AXIOM [4] place a point $S \notin \{ B, C \}$ in $\{ B, C \}^{\h \h}$. Theorem \ref{!plane} passes a unique plane $\pi$ through the non-collinear points $P, B$ and $C$; note that $\{ P, S, C \}^{\h} = \{ P, B, S \}^{\h} = \{ \pi \}$ also. AXIOM [2] provides a plane $\sigma \ne \pi$ in $\{ P, S \}^{\h}$; from $\alpha \notin P^{\h} \ni \sigma$ it follows that $\alpha \ne \sigma$. AXIOM [2] gives $\{ \alpha, \sigma \}^{\h} \subseteq \alpha^{\h}$ at least two points besides $S$; none of these can lie in $\{ B, C \}^{\h \h}$, for that would force $\sigma = \pi$. 
\end{proof} 

\medbreak 

In other words, if a line $\ell$  lies in a plane $\alpha$ then $\alpha$ passes not only through each point of $\ell$ but also through other points. As a complementary result, if the line $\ell$ does not lie in the plane $\alpha$, then $\ell$ and $\alpha$ pass through a unique common point; this is essentially the dual to Theorem \ref{!plane}. 

\medbreak 

This concludes what we wish to say here about our system itself. In particular, we refrain from further consideration of redundancies among our axioms; we shall not address fully the matter of independence. Regarding the matter of consistency, we now proceed to relate our axiomatic framework for three-dimensional projective space to the traditional Veblen-Young presentation [VY]. 

\medbreak 

Recall that the Veblen-Young approach is based on points and lines, in terms of which planes are defined as follows: if point $P$ and line $\ell$ are not incident, then the points of the plane $P \ell$ are precisely those points that lie on the lines that join $P$ to points of $\ell$; this is consistent with our framework. 

\medbreak 

\begin{theorem}
Let $\pi$ be the unique plane through the non-collinear points $A, B, C$. Then $P \in \pi^{\h}$ precisely when $P$ lies on $CD$ for some point $D$ on $A B$. 
\end{theorem}

\begin{proof} 
($\Rightarrow$) Let $P \in \pi^{\h}$: if $P = C$ then there is nothing to do; if $P \ne C$ then $A B$ and $C P$ have a point $D$ in common by Theorem \ref{meet}. ($\Leftarrow$) If $D \in \{ A, B \}^{\h \h}$ then $\{ C, D \}^{\h \h} \subseteq \pi^{\h}$ as in Theorem \ref{proper}; thus, if $P$ is a point of $CD$ then $P \in \pi^{\h}$. 
\end{proof} 

\medbreak 

Now, we claim that the Veblen-Young axioms of alignment and extension are theorems in our axiomatic framework. In the following discussion we take these Veblen-Young axioms in order,  offering a proof within our framework for each of them. Our labelling of these axioms follows [VY]; our phrasing of them differs inessentially from that in [VY]. 

\medbreak 

Consider first the axioms of alignment. 

\medbreak 

\noindent
AXIOM (A1): If $A$ and $B$ are distinct points, there is at least one line on both $A$ and $B$. 

\medbreak 

[Indeed, $A$ and $B$ lie on the line $A B = \{ A, B \}^{\h \h} \cup \{ A, B \}^{\h}$.] 

\medbreak 

\noindent
AXIOM (A2): If $A$ and $B$ are distinct points, there is at most one line on both $A$ and $B$.

\medbreak 

[Our definition of lines shows that if also $A$ and $B$ lie on $CD$ then $A B = C D$.] 

\medbreak 

\noindent
AXIOM (A3): If $A, B, C$ are points not on the same line, the line joining distinct points $D$ (on the line $B C$) and $E$ (on the line $C A$) meets the line $A B$. 

\medbreak 

[$\{ A, B, C \}^{\h}$ is a singleton $\{ \pi \}$ according to Theorem \ref{!plane}. Note that $D \in \{ B, C \}^{\h \h} \subseteq \pi^{\h}$ and $E \in \{ C, A \}^{\h \h} \subseteq \pi^{\h}$. Thus the lines $A B$ and $D E$ have the plane $\pi$ in common and so they have a point in common by Theorem \ref{meet}.]

\medbreak 

Consider next the axioms of extension. 

\medbreak 

\noindent
AXIOM (E0): There are at least three points on every line. 

\medbreak 

[This is immediate from AXIOM [2].]

\medbreak 

\noindent
AXIOM (E1): There exists at least one line. 

\medbreak 

[$\PP$ has at least three points; now summon AXIOM (A1).]  

\medbreak 

\noindent
AXIOM (E2): All points are not on the same line. 

\medbreak 

[This follows at once from the next axiom. See also Theorem \ref{proper}.] 

\medbreak 

\noindent 
AXIOM (E3): All points are not on the same plane. 

\medbreak 

[If $\pi \in \PPi$ is any plane then there exists $P \in \PP \setminus \pi^{\h}$ by AXIOM [1].]

\medbreak 

\noindent
AXIOM (E3$'$): Any two distinct planes have a line in common. 

\medbreak 

[If $\alpha \ne \beta$ in $\PPi$ then $\alpha \beta$ is a line incident to both.]

\bigbreak 

Thus, each of these Veblen-Young axioms is indeed a theorem of our axiomatic framework. In the opposite direction, we claim that each of our axioms is a theorem of the Veblen-Young axiomatic framework. The detailed verification of this claim is left as an exercise; most of the ingredients already appear as theorems and corollaries in Chapter 1 of [VY]. We also leave as an exercise the pleasure of directly relating the self-dual formulation (with line as the {\it secondary} element) presented here to the self-dual formulation (with line as the {\it primary} element) presented in [R]. 

\bigbreak

\begin{center} 
{\small R}{\footnotesize EFERENCES}
\end{center} 
\bigbreak 

[R] P. L. Robinson, {\it Projective Space: Lines and Duality}, arXiv 1506.06051 (2015). 

\medbreak 

[VY] O. Veblen and J.W. Young, {\it Projective Geometry}, Volume I, Ginn and Company, Boston (1910).

\medbreak

\end{document}